\documentclass[10pt,a4paper]{amsart}

\usepackage{amssymb}
\usepackage{amsfonts,amscd,mathabx,enumerate,calc}
\usepackage{amscd,amsopn,amsmath,verbatim}
\usepackage[dvips]{graphicx}
\usepackage[colorlinks=true,linkcolor=red,citecolor=blue]{hyperref}
\input xy
\xyoption{all}
\calclayout

\usepackage{geometry}
\usepackage{microtype}
\usepackage[autostyle]{csquotes}
\usepackage{mathrsfs}
\usepackage{latexsym}
\usepackage{rotating}
\usepackage{xspace}
\usepackage[all]{xy}
\usepackage{longtable}
\usepackage{leftidx}
\usepackage{mathtools}
\usepackage{titlesec}
\usepackage{tikz}
\usetikzlibrary{calc}
\usetikzlibrary{arrows}
\usetikzlibrary{decorations.markings}

\newcommand{\RNum}[1]{\uppercase\expandafter{\romannumeral #1\relax}}

\newcommand{\rt}{\rightarrow}

\newcommand{\st}{\stackrel}

\newcommand{\CC}{\mathcal{C} }

\newcommand{\CX}{\mathcal{X} }
\newcommand{\CY}{\mathcal{Y} }

\newcommand{\rmod}{{\rm{mod\mbox{-}}}}

\newcommand{\op}{{\rm{op}}}

\newcommand{\add}{{\rm{add}}}
\newcommand{\Add}{{\rm{Add}}}
\newcommand{\Prod}{{\rm{Prod}}}

\newcommand{\gen}{{\rm{gen}}}
\newcommand{\Gen}{{\rm{Gen}}}

\newcommand{\cogen}{{\rm{cogen}}}
\newcommand{\Cogen}{{\rm{Cogen}}}

\newcommand{\cok}{{\rm{Coker}}}
\newcommand{\Ker}{{\rm{Ker}}}

\newcommand{\Tor}{{\rm{Tor}}}
\newcommand{\Hom}{{\rm{Hom}}}
\newcommand{\Ext}{{\rm{Ext}}}
\newcommand{\End}{{\rm{End}}}

\newtheorem{theorem}{Theorem}[section]
\newtheorem{corollary}[theorem]{Corollary}
\newtheorem{lemma}[theorem]{Lemma}
\newtheorem{facts}[theorem]{Facts}

\theoremstyle{definition}
\newtheorem{definition}[theorem]{Definition}
\newtheorem{example}[theorem]{Example}

\newtheorem{remark}[theorem]{Remark}

\theoremstyle{plain}

\theoremstyle{definition}

\makeatletter
\newcommand{\holim@}[2]{%
\vtop{\m@th\ialign{##\cr
\hfil$#1\operator@font holim$\hfil\cr
\noalign{\nointerlineskip\kern1.5\ex@}#2\cr
\noalign{\nointerlineskip\kern-\ex@}\cr}}%
}
\newcommand{\holim}{%
\mathop{\mathpalette\holim@{\rightarrowfill@\textstyle}}\nmlimits@
}
\newcommand*{\medcap}{\mathbin{\scalebox{1.5}{\ensuremath{\cap}}}}%
\makeatother

\makeatletter
\def\@secnumfont{\bfseries}
\def\section{\@startsection{section}{1}%
\z@{.7\linespacing\@plus\linespacing}{.5\linespacing}%
{\normalfont\Large\bfseries\filcenter}}
\def\subsection{\@startsection{subsection}{2}%
\z@{.5\linespacing\@plus.7\linespacing}{-.5em}%
{\normalfont\large\bfseries}}
\makeatother
\begin{document}

\author[K. Divaani-Aazar, A. Mahin Fallah and M. Tousi]
{Kamran Divaani-Aazar, Ali Mahin Fallah and Massoud Tousi}

\title[Equivalence classes of ...]
{Equivalence classes of Wakamatsu (co)tilting modules and preenveloping and precovering subcategories}

\address{K. Divaani-Aazar, Department of Mathematics, Faculty of Mathematical Sciences, Alzahra University,
Tehran, Iran.}
\email{kdivaani@ipm.ir}

\address{A. Mahin Fallah, School of Mathematics, Institute for Research in Fundamental Sciences (IPM),
P.O. Box: 19395-5746, Tehran, Iran.}
\email{amfallah@ipm.ir, ali.mahinfallah@gmail.com}

\address{M. Tousi, Department of Mathematics, Faculty of Mathematical Sciences, Shahid Beheshti University,
Tehran, Iran.}
\email{mtousi@ipm.ir}


\keywords{Coresolving subcategory; covariantly finite subcategory; Ext-injective cogenerator; Ext-injective module;
Ext-projective generator; Ext-projective module; invertible module; resolving subcategory; right Morita ring;
self-orthogonal module; tilting module; Wakamatsu cotilting module; Wakamatsu tilting module.\\
The research of the second author is supported by a grant from IPM (No.1405130015).}

\begin{abstract}
In 1991, over an Artin algebra $A$, Auslander and Reiten established a one-to-one correspondence between the
isomorphism classes of basic tilting left $A$-modules and certain covariantly finite coresolving subcategories
of $A\text{-mod}$. They also obtained a one-to-one correspondence between the isomorphism classes of basic
cotilting modules and certain contravariantly finite resolving subcategories. Later, in 2004, Mantese and Reiten
extended these correspondences to basic Wakamatsu tilting left $A$-modules. In this paper, we generalize the
Mantese–Reiten results to arbitrary associative rings.

For an associative ring $R$, we introduce an equivalence relation ${\sim}$ on the class of Wakamatsu tilting right
$R$-modules. We establish a bijection between the equivalence classes $[T]$ and a family of preenveloping coresolving
subcategories of $\text{Mod}\text{-}R$, as well as a bijection between the classes $[T]$ and a family of resolving
subcategories of $\text{mod}\text{-}R$. Furthermore, we show that if $R$ is a Noetherian algebra over a Noetherian
commutative semi-local complete ring, then for two basic Wakamatsu tilting right $R$-modules $T$ and $T'$, one has
$T\sim T'$ if and only if $T\cong T'$. Consequently, the bijections obtained here extend the Mantese–Reiten theorems
from the setting of Artin algebras to that of arbitrary associative rings.
\end{abstract}

\maketitle

\tableofcontents

\section{Introduction}

Tilting theory plays a significant role in representation theory. Classical tilting modules over finite-dimensional
algebras were introduced in the early 1980s by Brenner and Butler \cite{BB}, and independently by Happel and Ringel
\cite{HR}. Miyashita \cite{M} studied tilting modules over general associative rings, while Wakamatsu \cite{W2}
extended the notion of tilting modules to include those that do not necessarily have finite projective dimension.
Green, Reiten, and Solberg \cite{GRS} referred to these modules as Wakamatsu tilting modules.

Let $A$ be an Artin algebra. The notions of covariantly and contravariantly finite subcategories of $A\text{-mod}$
were introduced by Auslander and Smal{\O} in \cite{AS1, AS2} in order to study subcategories of $A\text{-mod}$ admitting
almost split sequences. The study of these notions was further developed in \cite{AR1} in connection with tilting and
cotilting left $A$-modules. Auslander and Reiten \cite{AR1} showed that the assignment $T \mapsto T^{\perp}$ yields a
one-to-one correspondence between isomorphism classes of basic tilting left $A$-modules and covariantly finite coresolving
subcategories of $A\text{-mod}$ in which every module admits a finite coresolution. In the same paper, they also proved
that the assignment $T \mapsto {}^{\perp}T$ establishes a one-to-one correspondence between isomorphism classes of basic
cotilting left $A$-modules and contravariantly finite resolving subcategories of $A\text{-mod}$ in which every module admits
a finite resolution. In \cite{MR}, Mantese and Reiten extended these correspondences to Wakamatsu tilting left $A$-modules,
generalizing the classical tilting and cotilting cases; see \cite[Theorems 2.10 and 2.12]{MR}. In this paper, we further
generalize the results of Mantese and Reiten to arbitrary right modules, not necessarily finitely generated, over arbitrary
associative rings.

Let $R$ be an associative ring with identity. We introduce an equivalence relation ${\sim}$ on the class of right
$R$-modules and establish the following results:

\begin{theorem}\label{1.1}
There is a one-to-one correspondence between the equivalence classes of Wakamatsu tilting right $R$-modules, under the
equivalence relation ${\sim}$, and the preenveloping $\overline{\text{coresolving}}$ subcategories of right $R$-modules
that possess an Ext-projective $\overline{\text{generator}}$ in $\gen^*(R)$, maximal among those with the same
Ext-projective $\overline{\text{generator}}$.
\end{theorem}

\begin{theorem}\label{1.1a}
Let $R$ be a right Morita ring. There is a one-to-one correspondence between the isomorphism classes of basic Wakamatsu
tilting right $R$-modules and the coresolving subcategories of finitely generated right $R$-modules that possess
an Ext-projective generator and are maximal among those having the same Ext-projective generator.
\end{theorem}

\begin{theorem}\label{1.3}
There is a one-to-one correspondence between the equivalence classes of finitely generated, product-complete, Wakamatsu
cotilting right $R$-modules, under the equivalence relation ${\sim}$, and the precovering $\overline{\text{resolving}}$
subcategories of right $R$-modules that admit a finitely generated, product-complete, Ext-injective
$\overline{\text{cogenerator}}$, which is maximal among those sharing the same Ext-injective $\overline{\text{cogenerator}}$.
\end{theorem}

\begin{theorem}\label{1.2}
Let $R$ be a right Noetherian ring. There is a one-to-one correspondence between the equivalence classes of Wakamatsu
tilting right $R$-modules, under the equivalence relation ${\sim}$, and the resolving subcategories of finitely generated
right $R$-modules that possess an Ext-injective cogenerator, maximal among those with the same Ext-injective cogenerator.
\end{theorem}

\begin{theorem}\label{1.5}
Let $R$ be a right Morita ring. There is a one-to-one correspondence between the isomorphism classes of finitely generated,
basic product-complete, Wakamatsu cotilting right $R$-modules and the resolving subcategories of finitely generated right
$R$-modules with a product-complete Ext-injective cogenerator, maximal among those with the same Ext-injective cogenerator.
\end{theorem}

(All notions appearing in the above theorems will be recalled in Section 2.)

The notions of covariantly and contravariantly finite subcategories are typically defined in the category
of finitely generated $R$-modules. For arbitrary $R$-modules, the terms enveloping and covering are
more commonly used for these concepts.

The article is organized as follows:

In Section 2, we review background material and fundamental results on Wakamatsu tilting modules, including their
associated classes $\overline{\CX_{T}}$, $\CX_{T}$, $\overline{\CY_{T}}$, and $\CY_{T}$ of right $R$-modules.

Section 3 examines Wakamatsu cotilting modules. By \cite[Proposition 2.2]{BS}, over an Artin algebra, Wakamatsu
tilting and cotilting modules coincide. Consequently, Mantese and Reiten considered only Wakamatsu tilting modules
in their results. However, this coincidence does not generally hold for arbitrary associative rings. Nevertheless,
we show that if $R$ is a Noetherian algebra over a Noetherian commutative semi-local complete ring, then a
finite-length right $R$-module $T$ is Wakamatsu cotilting if and only if it is Wakamatsu tilting; see Theorem
\ref{3.6}.

Section 4 is devoted to defining the equivalence relation ${\sim}$ mentioned earlier. We do this by using the notion
of invertible bimodules. Bass \cite{B} introduced invertible bimodules in the context of defining Picard groups of
noncommutative rings. Invertible bimodules are precisely those that induce Morita equivalences. In this section, we
also introduce rank-one projective bimodules and show that they coincide with invertible bimodules; see Theorem
\ref{4.9}. Then we show that if $R$ is a Noetherian algebra over a Noetherian commutative semi-local complete ring,
and $T$ and $T'$ are basic Wakamatsu tilting right $R$-modules, then $T \sim T'$ if and only if $T \cong T'$. Hence,
Theorems \ref{1.1}, \ref{1.2}, and \ref{1.3} generalize the Mantese–Reiten theorems for Artin algebras and basic
Wakamatsu tilting modules to the general setting.

Section 5 is devoted to proving Theorems \ref{1.1} and \ref{1.1a}; see Theorems \ref{5.1} and \ref{5.8}. In Section 6,
we prove Theorems \ref{1.3}, \ref{1.2} and \ref{1.5}; see Theorems \ref{6.9}, \ref{6.1} and \ref{6.1a}.

\section{Preliminaries}

In this paper, we consider associative rings with identity, and all modules are assumed to be unitary. Let $R$ be an
associative ring with identity. We use the notation $M_R$ (respectively, $_RM$) to denote a right (respectively, left)
$R$-module. The category of all (respectively, finitely generated) right $R$-modules is denoted as $\text{Mod-}R$
(respectively, $\text{mod-}R$), while $R\text{-Mod}$ (respectively, $R\text{-mod}$) represents the category of all
(respectively, finitely generated) left $R$-modules.

Let $T_R$ be an $R$-module. We denote by $\Add(T_R)$ (respectively, $\add(T_R)$) the class of right $R$-modules which
are isomorphic to a direct summand of a direct sum of copies (respectively, finitely many copies) of $T$. Also,
$\Prod(T_R)$ stands for the class of right $R$-modules which are isomorphic to a direct summand of a direct product of
copies of $T$.

Denote by $\Gen^*(T)$ (respectively, $\gen^*(T)$), the class of all $R$-modules $M_R$ for which there exists an exact
sequence of the form $$\cdots \st{f_2}\longrightarrow T_1 \st{f_1}\longrightarrow T_0\st{f_0}\longrightarrow M
\longrightarrow 0$$ with each $T_i\in \Add(T_R)$ (respectively, $T_i\in \add(T_R)$) and $\Ext^1_R(T,\Ker \ f_i)=0$ for
all $i\geqslant 0$. Dually, $\Cogen^*(T)$ (respectively, $\cogen^*(T)$) is the class of all $R$-modules $M_R$ for which
there exists an exact sequence of the form $$0\longrightarrow M\st{f^{-1}}\longrightarrow
T^0\st{f^0}\longrightarrow T^1 \st{f^1}\longrightarrow \cdots$$ with each $T^i\in \Prod(T_R)$ (respectively, $T^i\in
\add(T_R)$) and $\Ext^1_R(\cok \ f^i,T)=0$ for all $i\geqslant -1$.

We need also to extend the definition of $\gen^*(T)$ to subcategories of $\text{Mod-}R$. For a subcategory $\CC$
of $\text{Mod-}R$, let $\gen^*(\CC)$ denote the class of all $R$-modules $M_R$ for which there exists an exact
sequence of the form $$\cdots\st{f_2}\longrightarrow T_1\st{f_1}\longrightarrow T_0 \st{f_0}\longrightarrow M
\longrightarrow 0$$ with each $T_i\in \CC$ and $\Ext^1_R(X,\Ker \ f_i)=0$ for every $X\in \CC$ and all $i\geqslant
0$.

The notions $\Add(-), \add(-), \Prod(-), \Gen^*(-), \gen^*(-), \Cogen^*(-)$ and $\cogen^*(-)$ for left $R$-modules
are all defined analogously.

Let $\CC$ be a subcategory of $\text{Mod-}R$. Then $\CC$ is said to be \emph{$\overline{coresolving}$} (respectively,
\emph{coresolving}) if it is closed under direct summands, extensions, and cokernels of monomorphisms,
and if it contains all injective modules (respectively, all finitely generated injective modules). Similarly, $\CC$
is said to be \emph{$\overline{resolving}$} (respectively, \emph{resolving}) if it is closed under direct summands,
extensions, and kernels of epimorphisms, and if it contains all projective modules (respectively, all finitely generated
projective modules).

The notation $\CC^{\perp}$ (respectively, $\CC^{\perp_{1}}$) denotes the subcategory of $R$-modules $M_R$
such that $\Ext^{i\geqslant1}_{R}(X,M)=0$ (respectively, $\Ext^1_{R}(X,M)=0$) for all $X\in \CC$. Similarly, the
notation $^{\perp}\CC$ (respectively, $^{\perp_{1}}\CC$) refers to the subcategory of $R$-modules $M_R$
such that $\Ext^{i\geqslant 1}_{R}(M,X)=0$ (respectively, $\Ext^1_{R}(M,X)=0$) for all $X\in \CC$. An $R$-module
$T_R\in \CC$ is said to be {\it Ext-projective} in $\CC$ if it belongs to $^{\perp}\CC$.
Furthermore, $T_R$ is called an {\it Ext-projective $\overline{generator}$} (respectively, {\it Ext-projective generator})
for $\CC$ if it is an Ext-projective module in $\CC$ and for any module $M\in \CC$, there exists
an exact sequence $$0\rightarrow M' \rightarrow T'\rightarrow M\rightarrow 0$$ with $T'\in \Add(T_R)$ (respectively, $T'\in
\add(T_R)$) and $M'\in \CC$. Dually, an $R$-module $T_R\in \CC$ is said to be {\it Ext-injective} in
$\CC$ if it belongs to $\CC^{\perp}$. It is called an {\it Ext-injective $\overline{cogenerator}$}
(respectively, {\it Ext-injective cogenerator}) for $\CC$ if it is an Ext-injective module in $\CC$ and for any module
$M\in \CC$, there exists an exact sequence $$0\rightarrow M \rightarrow T'\rightarrow M'\rightarrow 0$$ with
$T'\in \Prod(T_R)$ (respectively, $T'\in \add(T_R)$) and $M'\in \CC$.

Recall that an $R$-module $M_R$ is called {\it self-orthogonal} if $\Ext^{i\geq 1}_{R}(M,M)=0$. Next, we recall the
definition of Wakamatsu tilting modules from \cite[Section 3]{W1}.

\begin{definition}\label{2.1}
A self-orthogonal module $T_R$ is called a {\it Wakamatsu tilting module} if
\begin{itemize}
\item[(i)] $T_R\in \gen^*(R)$, and
\item[(ii)] $R_R \in \cogen^* (T)$.
\end{itemize}
\end{definition}

Wakamatsu tilting left modules are defined similarly. By \cite[Corollary 3.2]{W1}, we have the following characterization
of the Wakamatsu tilting modules.

\begin{lemma}\label{2.2}
For a bimodule $_ST_R$, the following are equivalent:
\begin{itemize}
\item[(i)] $T_R$ is a Wakamatsu tilting module with $S\cong \End(T_R);$
\item[(ii)] $_ST$ is a Wakamatsu tilting module with $R\cong \End(_ST);$
\item[(iii)] One has
\begin{itemize}
\item[(1)]  $T_R\in \gen^*(R)$ and $_ST\in \gen^*(S)$.
\item[(2)]  $S\cong \End(T_R)$ and $R\cong \End(_ST)$.
\item[(3)] The modules $T_R$ and $_ST$ are self-orthogonal.
\end{itemize}
\end{itemize}
\end{lemma}

From now on, when we say that $_S T_R$ is a Wakamatsu tilting bimodule, we mean that $T_R$ is a Wakamatsu tilting module
and $S \cong \End(T_R)$.

In the next two definitions, we recall the notions of preenvelopes, precovers, and cotorsion pairs.

\begin{definition}\label{2.3}
Let $\mathscr{C}$ be a subcategory of $\text{Mod-}R$, and let $M_R$ be a module.
\begin{itemize}
\item[(i)] An $R$-homomorphism $\phi: M\rt C_M$ with $C_M\in \mathscr{C}$ is called a {\it $\mathscr{C}$-preenvelope} of
$M$ if the induced morphism $\Hom_R(C_M,C)\rt \Hom_R(M,C)$ is surjective for every $C\in \mathscr{C}$.  If, in addition,
$\phi$ is injective and $\cok \ \phi \in {}^{\perp_{1}}\mathscr{C}$, then $\phi$ is called a {\it special
$\mathscr{C}$-preenvelope} of $M$.
\item[(ii)] The subcategory $\mathscr{C}$ is called {\it preenveloping} if every right $R$-module admits a $\mathscr{C}$-preenvelope.
\item[(iii)]  An $R$-homomorphism $\phi: C_M\rt M$ with $C_M\in \mathscr{C}$ is called a {\it $\mathscr{C}$-precover} of
$M$ if the induced morphism $\Hom_R(C,C_M)\rt \Hom_R(C,M)$ is surjective for every $C\in \mathscr{C}$. If, in addition,
$\phi$ is surjective and $\Ker \ \phi \in \mathscr{C}^{\perp_{1}}$, then $\phi$ is called a {\it special $\mathscr{C}$-precove}
of $M$.
\item[(iv)] The subcategory $\mathscr{C}$ is called {\it precovering} if every right $R$-module admits a $\mathscr{C}$-precover.
\end{itemize}
\end{definition}

\begin{definition}\label{2.4}
Let $\CC$ and $\mathcal{D}$ be two subcategories of $\text{Mod-}R$.
\begin{itemize}
\item[(i)] The pair $(\CC,\mathcal{D})$ is called a {\it cotorsion pair} if $\CC={^{\perp_1}}\mathcal{D}$ and $\mathcal{D}=
\CC^{\perp_1}$.
\item[(ii)] A cotorsion pair $(\CC,\mathcal{D})$ is called {\it complete} if every right $R$-module has a special $\CC$-precover
(or, equivalently by \cite[Corollary~2.4]{Sa}, every right $R$-module has a special $\mathcal{D}$-preenvelope).
\end{itemize}
\end{definition}

We recall the following result from \cite[Corollary 4.5]{An}.

\begin{lemma}\label{2.5}
Let $M_R$ be a module. Then $(^\perp(M^{\perp}), M^{\perp})$ is a complete cotorsion pair.
\end{lemma}

Let $T_R$ be a module. We define $\overline{\CX_{{T}}}={T}^\perp \medcap \Gen^*(T)$ and $\CX_{{T}}={T}^\perp
\medcap \gen^*(T)$. It is immediate that if $T_R$ is finitely generated, then $\CX_{T}$ is a full subcategory
of $\text{mod-}R$. The following result summarizes some well-known properties of the classes $\overline{\CX_{{T}}}$
and $\CX_{{T}}$ in the case where $T_R$ is a Wakamatsu tilting module.

\begin{lemma}\label{2.6}
Let $_ST_R$ be a Wakamatsu tilting bimodule. Then
\begin{itemize}
\item[(i)] $M_R\in \overline{\CX_{T}}$ if and only if $\Ext^{i\geqslant1}_R(T,M)=0$, $\Tor^S_{i\geqslant1}(\Hom_R(T,M),T)=0$
and the natural map $$\nu_M: \Hom_R(T,M)\otimes_ST\to M$$ is an isomorphism.
\item[(ii)] there is a module $N_R$ such that $\overline{\CX_{T}}=N^\perp$.
\item[(iii)] $\overline{\CX_{T}}$ is a preenveloping $\overline{coresolving}$ subcategory of $\text{Mod-}R$.
\item[(iv)]  $T_R$ is an Ext-projective $\overline{generator}$ for $\overline{\CX_{T}}$. Moreover, if $\CC$ is a subcategory of
$\text{Mod-}R$ having $T_R$ as an Ext-projective $\overline{generator}$, then $\CC\subseteq \overline{\CX_{T}}$.
\item[(v)]  $\overline{\CX_{T}}\medcap \gen^*(R)=\CX_{T}$.
\end{itemize}
\end{lemma}

\begin{proof}
(i) This follows from \cite[Proposition 3.1(1)]{S}.

(ii) The claim follows from (i) and \cite[Theorem 3.4(1)]{S}.

(iii) The result follows from (ii) and Lemma \ref{2.5}.

(iv) Since $T_R$ is finitely presented, for any family $\{X_{\gamma}\}_{\gamma\in \Gamma}$ of right $R$-modules and every
integer $i\geq 0$, there is an isomorphism $\Ext_R^i(T,\bigoplus_{\gamma\in \Gamma} X_\gamma)\cong \bigoplus_{\gamma
\in \Gamma} \Ext_R^i(T, X_\gamma).$ Consequently, $\Add(T_R)\subseteq T^{\perp}$. The assertion in (iv) now follows directly
from the definition.

(v) This holds by \cite[Proposition 3.1(1) and Corollary 3.2(1)]{S}.
\end{proof}

\begin{lemma}\label{2.6a}
Let $_ST_R$  be a Wakamatsu bimodule.
\begin{itemize}
\item[(i)] $\mathcal{X}_T$ is closed under direct summands, extensions, and cokernels of monomorphisms.
\item[(ii)] Let $R$ be right Noetherian. Then $\mathcal{X}_T$ is a coresolving subcategory of $\text{mod-}R$.
\item[(iii)]$T_R$ is an Ext-projective $generator$ for $\CX_{T}$. Moreover, if $\CC$ is a subcategory of
$\text{mod-}R$ having $T_R$ as an Ext-projective $generator$, then $\CC\subseteq \CX_{T}$.
\end{itemize}
\end{lemma}

\begin{proof}
(i) By \cite[Lemma 2.2]{W1}, $\gen^*(T)$ is closed under direct summands. Since $\CX_{T}=T^\perp \medcap \gen^*(T)$,
it follows immediately that $\CX_{T}$ is also closed under direct summands. Moreover, \cite[Lemma 2.3(2)]{W1}
implies that $\CX_{T}$ is closed under extensions.

Next, we show that $\mathcal{X}_T$ is closed under cokernels of monomorphisms. Consider the exact sequence
$$0\rt X\rt Y \rt Z\rt 0$$ with $X, Y \in \mathcal{X}_T$. Since $X$ and $Y$ lie in $T^\perp$, from this
exact sequence, we see that $Z$ also belongs to $T^\perp$. Therefore, it remains to verify that $Z\in \gen^*(T)$.

Since $Y\in \gen^*(T)$,
we have an exact sequence $$0 \rt Y_1 \rt T_0 \rt Y \rt 0,$$ where $T_0\in \add(T)$, $Y_1\in \gen^*(T)$ and
$\Ext^1_R(T,Y_1)=0$. Now, we can construct the following pull-back diagram:
\[\xymatrix@C-0.5pc@R-.8pc{ \\ &  0\ar[d]&0\ar[d] &&\\& Y_1\ar[r]^{\text{id}_{Y_1}} \ar[d] &Y_1  \ar[d]&&\\
0\ar[r] & Z^* \ar[r]\ar[d]  & T_0 \ar[d] \ar[r] & Z  \ar[d]^{\text{id}_{Z}}\ar[r]&0&\\0\ar[r]& X
\ar[r] \ar[d] & Y \ar[d]\ar[r] & Z  \ar[r]&
0& \\  &0&0  &&  }\]
As $Y_1, X \in \gen^*(T)$ and $\Ext^1_R(T,Y_1)=0$, applying \cite[Lemma 2.3(2)]{W1} to the left column of this
diagram, yields that  $Z^*\in \gen^*(T)$. Since   $\Ext^1_R(T,Y_1)=0=\Ext^1_R(T,X)$,  the left column of the
above diagram implies that $\Ext^1_R(T,Z^*)=0$. So, by the middle row of the above diagram, we conclude that
$Z\in \gen^*(T)$.

(ii) As $R$ is right Noetherian, every finitely generated module $M_R$ belongs to $\gen^*(R)$. By parts
(ii) and (v) of the above lemma, there exists a module $N_R$ such that $\CX_{T}=N^\perp \medcap \gen^*(R)$.
Hence, $\mathcal{X}_T$ contains all finitely generated injective right $R$-modules. Therefore, by (i), $\CX_{T}$
is a coresolving subcategory of $\text{mod-}R$.

(iii) The statement follows immediately from the definition.
\end{proof}

For the following definition, we refer the reader to \cite{AR2} and \cite{H}.

\begin{definition}\label{2.7}
Let $_ST_R$ be a Wakamatsu tilting bimodule. A module $M_R$ is said to have {\it generalized Gorenstein dimension zero}
with respect to $T$,  if the following conditions are satisfied:
\begin{itemize}
\item[(i)] $M\in \gen^*(R)$.
\item[(ii)] $\Hom_R(M,T)\in \gen^*(S)$.
\item[(iii)] $\Ext^{i\geqslant1}_R(M,T)=0$.
\item[(iv)] $\Ext^{i\geqslant1}_S(\Hom_R(M,T),T)=0$.
\item[(v)] The natural map $\theta_M: M\to \Hom_S(\Hom_R(M,T),T)$ is an isomorphism.
\end{itemize}
We use $\mathcal{G}_T$ to denote the full subcategory of $\text{mod-}R$ consisting of the modules with generalized
Gorenstein dimension zero with respect to $T$.
\end{definition}

Next, we recall the definition of Wakamatsu cotilting modules. In Section 3, we will examine these modules in more detail.

\begin{definition}\label{2.8}
A module $C_R$ is called a {\it Wakamatsu cotilting module} if the following conditions hold:
\begin{itemize}
\item[(i)] $\Ext^{i}_R(C^I,C)=0$ for every set $I$ and all $i\geqslant 1$.
\item[(ii)] There exists an injective cogenerator $Q_R$ for $\text{Mod-}R$ in $\gen^*(\Prod(C_R))$.
\end{itemize}
\end{definition}

Let $T_R$ be a module. We set $\overline{\CY_{T}}={}^\perp{T}\medcap \Cogen^*(T)$ and $\CY_{T}={}^\perp{T} \medcap \cogen^*(T)$.
If $R$ is right Noetherian and $T_R$ is finitely generated, then clearly  $\CY_{T}$ is a full subcategory of $\text{mod-}R$.
In the next two lemmas, we summarize some well-known properties of the classes $\overline{\CY_{T}}$ and $\CY_{T}$.

\begin{lemma}\label{2.9}
Let $C_R$ be a Wakamatsu cotilting module. Then
\begin{itemize}
\item[(i)] there is a module $N_R$ such that $\overline{\CY_{C}}= ~^\perp N$.
\item[(ii)]  $\overline{\CY_{C}}$ is a precovering $\overline{resolving}$ subcategory of $\text{Mod-}R$.
\item[(iii)] $C_R$ is an Ext-injective $\overline{cogenerator}$ for $\overline{\CY_{C}}$. Moreover, if $\CC$ is a
subcategory of $\text{Mod-}R$ having $C_R$ as an Ext-injective $\overline{cogenerator}$, then $\CC\subseteq
\overline{\CY_{C}}$.
\end{itemize}
\end{lemma}

\begin{proof}
\begin{itemize}
\item[(i)]  By \cite[Lemma 2.1]{S}, $\Prod(C_R)$ is a preenveloping subcategory of $\text{Mod-}R$. Also, by definition
of Wakamatsu cotilting module, there exists an injective cogenerator $Q$ such that $Q\in \gen^*(\Prod(C_R))$. So, the
desired assertion follows from \cite[Lemma 3.3(1)]{S}.
\item[(ii)]  By $(i)$, it is easy to verify that the subcategory $\overline{\CY_{C}}$ is $\overline{resolving}$. So, we
need only to show that $\overline{\CY_{C}}$ is precovering. To do this, by $(i)$ there is a module $N_R$ such that
$\overline{\CY_{C}}=~^\perp N$. On the other hand, it is easy to see that $~^\perp((^\perp N)^\perp)=~^\perp N$. Put
$M=\bigoplus_{X\in ~ ^\perp N}X$. Then, we have $M^{\perp}=(^\perp N)^\perp$. Thus, $~^\perp(M^{\perp})=~^\perp N$.
Consequently, by Lemma \ref{2.5}, $(^\perp N, M^{\perp})$ is a complete cotorsion pair. Hence, $\overline{\CY_{C}}$
is precovering.
\item[(iii)] In view of condition (i) in Definition \ref{2.8}, this is obvious by the definition of $\overline{\CY_{C}}$.
\end{itemize}
\end{proof}

\begin{lemma}\label{2.10}
Let $T_R$ be a module such that $\Ext^1_R(T,T)=0$.
\begin{itemize}
\item[(i)] ${\CY_{T}}$ is closed under direct summands, extensions, and kernels of epimorphisms.
\item[(ii)]  If $R\in \cogen^*(T)$, then  ${\CY_{T}}$ is a resolving  subcategory of $\text{Mod-}R$.
\item[(iii)] Assume that $T_R$ is self-orthogonal. Then $T_R$ is an Ext-injective cogenerator for ${\CY_{T}}$.
Moreover, if $\CC$ is a subcategory of $\text{Mod-}R$ that has $T_R$ as an Ext-injective cogenerator, then $\CC
\subseteq {\CY_{T}}$.
\item[(iv)] If $T_R$ is a Wakamatsu tilting module, then $\CY_{T}=\mathcal{G}_T$.
\end{itemize}
\end{lemma}

\begin{proof}
(i) By \cite[Lemma 2.3(1)]{W1} and \cite[Lemma 2.2]{W1}, we conclude that ${\CY_{T}}$ is closed under direct
summands and extensions. Furthermore, by an argument dual to that used in the proof of Lemma \ref{2.6a}(i), we
obtain that ${\CY_{T}}$ is also closed under kernels of epimorphisms.

(ii) Since $R \in \cogen^*(T)$, it follows that $R \in \CY_{T}$. Hence, by (i), $\CY_{T}$ contains all finitely
generated projective right $R$-modules. In view of (i), this shows that $\CY_{T_R}$ is a resolving subcategory
of $\text{Mod-}R$.

(iii) is obvious.

(iv) The assertion follows from \cite[Propositions 1 and 3]{W2}, whose arguments remain valid for any associative ring.
\end{proof}

We conclude the preliminaries by recalling the well-known projectivization theorem, which can be readily verified. For a
detailed proof in the case when $R$ is an Artin algebra, see, for example, \cite[{\bf II}, Proposition 2.1]{ARS}.

\begin{lemma}\label{2.11}
Let $M\in \text{mod-}R$ and $S=\End(M_R)$.
\begin{itemize}
\item[(i)]	The functor $\Hom_R(M,-):\text{mod-}R\longrightarrow \text{mod-}S$ induces an equivalence of categories between
$\add(M_R)$ and the category of finitely generated projective right $S$-modules.
\item[(ii)] The functor $\Hom_R(-,M):\text{mod-}R\longrightarrow S\text{-mod}$ induces a duality of categories between
$\add(M_R)$ and the category of finitely generated projective left $S$-modules.
\end{itemize}
\end{lemma}

\section{Wakamatsu cotilting modules}

We begin this section by introducing the following natural generalization of the notion of Artin algebras.

\begin{definition}\label{3.1}
A ring $R$ is called a {\it complete algebra} if there exists a Noetherian commutative semi-local complete ring $A$ such
that $R$ is a Noetherian $A$-algebra.
\end{definition}

Over an Artin algebra $R$, by \cite[Proposition 2.2]{BS}, it is known that a finitely generated module $C_R$ is
Wakamatsu cotilting if and only if it is Wakamatsu tilting. We generalize this to any complete algebra.

\begin{lemma}\label{3.2}
Let $C_R$ be a Wakamatsu cotilting module and set $S=\End(C_R)$. Assume that the rings $R$ and $S$ are right and left coherent,
respectively, and that both modules $C_R$ and $_S C$ are finitely presented. Then $C_R$ is a Wakamatsu tilting module.
\end{lemma}

\begin{proof}
See \cite[Lemma 3.6]{DMT}.
\end{proof}

To present the main result of this section, we need the notion of product-complete modules. We now recall the
definition of this concept.

\begin{definition}\label{3.5}
A module $C_R$ is called {\it product-complete} if $\Prod(C_R)\subseteq \Add(C_R)$.
\end{definition}

As mentioned in \cite[Remark 3.5]{MR}, a module $C_R$ is product-complete if and only if $\Prod(C_R)=\Add(C_R)$. For a
module $C_R$, where the ring $S=\End(C_R)$ is left Noetherian, by \cite[Corollary 4.4]{KS}, $C_R$ is product-complete if
and only if $_SC$ has finite length.

\begin{theorem}\label{3.6}
Let $R$ be a complete algebra, and $T_R$ be a finite-length module. Then $T_R$ is a Wakamatsu cotilting module if and
only if it is a Wakamatsu tilting module.
\end{theorem}

\begin{proof}
Let $S=\End(T_R)$. Since $R$ is a complete algebra, there exists a Noetherian commutative semi-local complete ring $A$ and
a ring homomorphism $f:A\to R$ such that $R$ is finitely generated as an $A$-module via $f$. By \cite[Proposition 4.2(c),(d)]{A},
it follows that $S$ is also a complete $A$-algebra and $_ST$ also has finite length. Hence, by Lemma \ref{3.2}, it suffices
to show that if $T_R$ is Wakamatsu tilting, then it is also Wakamatsu cotilting.

Assume that $T_R$ is a Wakamatsu tilting module. Then $T_R$ is self-orthogonal, $T_R \in \gen^*(R)$, and $R_R \in \cogen^*(T)$.
Since $S$ is left Noetherian and $_S T$ has finite length, \cite[Corollary 4.4]{KS} yields that $T_R$ is product-complete.
Thus, for any set $I$, there is a set $J$ such that $T^I$ is isomorphic to a direct summand of $T^{(J)}$. For any $i>0$, as $$\Ext^i_R(T^{(J)},T)\cong \Ext^i_R(T,T)^J=0,$$ it follows that $\Ext^i_R(T^I,T)=0$.

Let $\mathscr{E}$ be the minimal injective cogenerator of $A$, and set $D(-)=\Hom_A(-,\mathscr{E})$. Clearly, $Q=D(R)$ is
an injective cogenerator for $\text{Mod-}R$. We can complete the argument by showing that $Q\in \gen^*(\Prod(T_R))$.

For each integer $i\geq 0$, it is straightforward to see that $$\Ext^i_R(D(X),D(Y))\cong \Ext^i_R(Y,X)$$ for all $X_R,
Y_R\in\gen^*(R)$, and $$\Ext^i_S(D(X),D(Y))\cong \Ext^i_S(Y,X)$$ for all $_SX, _SY\in\gen^*(S)$.

Since $_ST_R$ is a Wakamatsu tilting module, Lemma \ref{2.2} implies that $R\cong \End(_ST)$ and $_ST$ is self-orthogonal.
Also, we know that $S=\End(T_R)$ and $T_R$ is self-orthogonal. Hence, by the above isomorphisms, we obtain $S\cong \End(_RD(T))$,
$R\cong \End(D(T)_S)$, and $_RD(T)$ and $D(T)_S$ are self-orthogonal. Moreover, as $T_R$ and $_ST$ have finite length,
the modules $_RD(T)$ and $D(T)_S$ also have finite length. In particular, $_RD(T)\in \gen^*(R)$ and $D(T)_S\in \gen^*(S)$.
Thus, applying Lemma \ref{2.2} once again, we conclude that $_RD(T)_S$ is a Wakamatsu tilting module.

Since $_RD(T)_S$ is a Wakamatsu tilting module, we have $_RR \in \cogen^*(D(T))$. Thus, there exists an exact sequence
$$0\longrightarrow R\st{f^{-1}}\longrightarrow X^0\st{f^0}\longrightarrow X^1 \st{f^1}\longrightarrow \cdots$$ with each
$X^i\in \add(_RD(T))$ and $\Ext^1_R(\cok \ f^i,D(T))=0$ for all $i\geqslant -1$. Applying the functor $D$ to this exact
sequence yields $Q \in \gen^*(T)$. Since $T_R$ is product-complete, we have $T^{\perp_1}=\Prod(T_R)^{\perp_1}$, and thus
$\gen^*(T)\subseteq \gen^*(\Prod(T_R))$. Therefore, $Q \in \gen^*(\Prod(T_R))$.
\end{proof}

We conclude this section with the following immediate corollary.

\begin{corollary}\label{3.7}
Let $R$ ba an Artin algebra and let $T_R$ be a finitely generated module.  Then $T_R$ is a Wakamatsu cotilting
module if and only if it is a Wakamatsu tilting module.
\end{corollary}

\section{Invertible modules}

Theorem \ref{4.9} is the main result of this section. To present it, we first recall some definitions. Moreover, for
its proof, we will need Lemmas \ref{4.4} and \ref{4.6}. Let us begin with the following useful observation.

\begin{facts}\label{4.1}
Let $_SL_{R'}$, $_SM_R$ and $_{S'}N_R$ be three bimodules. The following hold:
\begin{itemize}
\item[(i)] $\Hom_R(M_R, N_R)$ is naturally an $(S',S)$-bimodule. Similarly, $\Hom_S({}_S L, {}_S M)$ is naturally an
$(R',R)$-bimodule.
\item[(ii)] The natural map $\eta_{M_R}:S\rt \End(M_R)$ is an $(S,S)$-bihomomorphism.
\item[(iii)] The natural map $\eta_{_SM}:R\rt \End(_SM)$ is an $(R,R)$-bihomomorphism.
\end{itemize}
\end{facts}

Next, we recall the following definition from \cite{AF}.

\begin{definition}\label{4.2}
A finitely generated projective module $P_R$ is called a {\it progenerator} if it is a generator for $\text{Mod-}R$,
i.e., for every module $M_R$, there exists an epimorphism from a direct sum of copies of $P$ onto $M$.
\end{definition}

Progenerator left $R$-modules are defined similarly. It can be shown that a module $P_R$ is a progenerator if and only
if there exist positive integers $m$ and $n$, and modules $P'_R$ and $R'_R$, such that $R^{(m)}\cong P\oplus P'$ and 
$P^{(n)}\cong R\oplus R'$.

We proceed to introduce the notion of projective bimodules of rank one. To this end, we first recall the notion of
faithfully balanced bimodules.

\begin{definition}\label{4.5}
A bimodule $_SM_R$ is called {\it faithfully balanced} if the natural maps $\eta_{M_R}:S\rt \End(M_R)$ and $\eta_{_SM}:R\rt
\End(_SM)$ are isomorphisms.
\end{definition}

\begin{definition}\label{4.3}
A faithfully balanced bimodule $_SP_R$ is called {\it projective of rank one} if both $P_R$ and $_SP$ are finitely generated
projective modules.
\end{definition}

\begin{lemma}\label{4.4}
Let $_SP_R$ be a projective bimodule of rank one. Then both modules $P_R$ and $_SP$ are progenerators.
\end{lemma}

\begin{proof}
Since the proofs of the claims for $_SP$ and $P_R$ are similar, we prove the claim only for $_SP$. As $_SP$ is a finitely
generated projective module, it suffices to show that $_SP$ is a generator for $S\text{-Mod}$. By \cite[Proposition 17.5]{AF},
equivalently, we need to show that for any homomorphism $f:X \rightarrow Y$ of left $S$-modules, if the induced homomorphism
$$\Hom_S(P,f):\Hom_S(P,X) \rightarrow \Hom_S(P,Y)$$ is zero, then $f=0$.

Let $_SX$ and $_SY$ be two modules, and let $f:X \rightarrow Y$ be an $S$-homomorphism such that $\Hom_S(P,f)=0$. For any
$S$-module $_SZ$, the Hom-evaluation map yields a natural isomorphism $$P \otimes_R \Hom_S(P,Z)\to \Hom_S(\Hom_R(P,P),Z).$$
Consider the following commutative diagram, in which all vertical maps are natural isomorphisms:
\begin{displaymath}
\xymatrix@C=22ex{
P \otimes_R \Hom_S(P,X) \ar@{->}[d] \ar@<0.4ex>[r]^-{P\otimes_R\Hom_S(P,f)} & P \otimes_R \Hom_S(P,Y) \ar@{->}[d] \\
\Hom_S(\Hom_R(P,P),X) \ar@{->}[d] \ar@<0.4ex>[r]^-{\Hom_S(\Hom_R(P,P),f)} & \Hom_S(\Hom_R(P,P),Y) \ar@{->}[d] \\
\Hom_S(S,X) \ar@{->}[d] \ar@<0.4ex>[r]^-{\Hom_S(S,f)} & \Hom_S(S,Y) \ar@{->}[d] \\
X \ar@<0.4ex>[r]^-{f} & Y}
\end{displaymath}
From this diagram, we clearly conclude that $f=0$.
\end{proof}

\begin{lemma}\label{4.6}
Let $R$ and $S$ be two rings and suppose $F:R\text{-Mod}\rightarrow S\text{-Mod}$ and $G:S\text{-Mod}\rightarrow R\text{-Mod}$
are additive covariant functors that establishing an equivalence of categories. Set $P=F(R)$ and $Q=G(S)$. Then $P$ and $Q$
are naturally $(S,R)$-bimodule and $(R,S)$-bimodule; respectively such that
\begin{itemize}
\item[(i)] $_SP_R$ and $_RQ_S$ are faithfully balanced;
\item[(ii)] $_SP$, $P_R$, $_RQ$ and $Q_S$ are all progenerators;
\item[(iii)] $F\cong P\otimes_R-$ and $G\cong Q\otimes_S-$;
\item[(iv)] $F\cong \Hom_R(Q,-)$ and $G\cong \Hom_S(P,-)$.
\end{itemize}
\end{lemma}

\begin{proof}
See \cite[Theorem 22.1]{AF}.
\end{proof}

We next recall the definition of invertible modules; see \cite[Chapter 2, Definition 3.2]{B}.

\begin{definition}\label{4.7}
A bimodule $_SP_R$ is called {\it invertible} if the functor $P\otimes_R-:R\text{-Mod}\rightarrow S\text{-Mod}$ is an
equivalence of categories.
\end{definition}

Indeed, the invertibility of a bimodule ${}_S P_R$ is equivalent to the existence of a bimodule ${}_R Q_S$ such that
$P \otimes_R Q \cong S$ as $(S,S)$-bimodules and $Q \otimes_S P \cong R$ as $(R,R)$-bimodules.

The following result can be easily verified.

\begin{lemma}\label{4.8}
Let $_SP_{S'}$ and $_{S''}Q_S$ be invertible bimodules. Then the tensor product $_{S''}(Q\otimes_SP)_{S'}$ is also an
invertible bimodule.
\end{lemma}

Although the assertion of the following result may be known to experts, we could not find a proof in the literature;
therefore, we provide one here.

\begin{theorem}\label{4.9}
For a bimodule $_SP_R$, the following statements are equivalent:
\begin{itemize}
\item[(i)] $_SP_R$ is invertible.
\item[(ii)] $_SP_R$ is projective of rank one.
\end{itemize}
\end{theorem}

\begin{proof}
(i)$\Rightarrow$(ii) Since $_SP_R$ is invertible, the functor $P \otimes_R-:R\text{-Mod} \rightarrow
S\text{-Mod}$ is an equivalence of categories. Applying Lemma \ref{4.6} to the functor $F = P \otimes_R -$ yields that
the bimodule $_SP_R \cong F(R)$ is faithfully balanced, and the modules $_SP$ and $P_R$ are progenerators. These imply
that $_SP_R$ is projective of rank one.

(ii)$\Rightarrow$(i) Since $_SP_R$ is a projective bimodule of rank one, it follows from the definition that $_SP_R$ is
faithfully balanced. Moreover, Lemma \ref{4.4} yields that both modules $P_R$ and $_SP$ are progenerators. Thus, by
\cite[Theorem 22.2]{AF}, we conclude that $P \otimes_R-:R\text{-Mod} \rightarrow S\text{-Mod}$ is an equivalence of
categories, and so $_SP_R$ is invertible.
\end{proof}

We are now in a position to introduce the equivalence relation mentioned in the introduction.

\begin{definition}\label{4.10}
We define an equivalence relation ${\sim}$ on  the class of all right $R$-modules as follows: $_ST_R~{\sim}~ _{S'}T'_R$
if and only if there exist two invertible bimodules $_SP_{S'}$ and $_{S'}Q_S$ such that $T\cong P\otimes_{S'}T'$ and
$T'\cong Q\otimes_ST$, as right $R$-modules.
\end{definition}

Next, we show that the relation already defined is indeed an equivalence relation.

\begin{lemma}\label{4.11}
The relation ${\sim}$ on the class of all right $R$-modules is an equivalence relation.
\end{lemma}

\begin{proof}
Clearly, the relation ${\sim}$ is reflexive and symmetric.

To show transitivity, assume that $_{S}T_R ~{\sim} ~_{S'}{T'}_R$ and $_{S'}{T'}_R~ {\sim}~ _{S''}T''_R$. By definition,
there exist invertible bimodules, $_SP_{S'}$, $_{S'}Q_S$, $_{S'}P'_{S''}$ and $_{S''}Q'_{S'}$, such that
\begin{align*}
T &\cong P\otimes_{S'}T', \\
T' &\cong  Q\otimes_ST, \\
T' &\cong P'\otimes_{S''}T'',\ \text{and} \\
T'' &\cong Q'\otimes_{S'}T'.
\end{align*}
Consequently, we have the following isomorphisms of right $R$-modules: $$T\cong P\otimes_{S'}(P'\otimes_{S''}T'')\cong (P\otimes_{S'}P')\otimes_{S''}T''$$ and $$T''\cong Q'\otimes_{S'}(Q\otimes_ST)\cong (Q'\otimes_{S'}Q)\otimes_ST.$$
By Lemma \ref{4.8}, the bimodules $_S(P\otimes_{S'} P')_{S''}$ and $_{S''}(Q'\otimes_{S'} Q)_{S}$  are invertible,
and so $_{S}{T}_R~ {\sim}~ _{S''}T''_R$. Therefore, the relation ${\sim}$ is transitive, and hence an equivalence
relation.
\end{proof}

The equivalence class of a module $T_R$ under $\sim$ is denoted by $[T]$.

\begin{example}\label{4.12}
It is somewhat inconvenient that the definition of the equivalence relation ${\sim}$ involves three rings. We mainly
consider this equivalence relation on the class of Wakamatsu tilting right $R$-modules. Thus, it is natural to ask:
if $T_R$ and $T'_R$ are two Wakamatsu tilting modules, is it true that $\End(T_R) \cong \End(T'_R)$? The answer is no.

In fact, let $k$ be a field, and let $R=kQ$ be the path algebra of the quiver $\mathcal{Q}: 1 \longleftarrow 2
\longleftarrow 3$, and let $T=P(1) \oplus P(3) \oplus S(3)$. It is easy to see that $T_R$ is a Wakamatsu tilting
module. By \cite[Example VI.3.11]{ASS}, $\End(T_R)$ is given by the quiver $1\st{\mu}\longleftarrow 2 \st{\lambda}
\longleftarrow 3$ bound by $\lambda \mu=0$.

On the other hand, if we take $T'_R=R$, then $T'_R$ is also a Wakamatsu tilting module and $\End(T'_R)\cong R$.
Hence, $\End(T_R)$ is not isomorphic to $\End(T'_R)$.
\end{example}

The next result provides an equivalent definition of the equivalence relation defined above.

\begin{lemma}\label{4.13}
For any two bimodules $_ST_R$ and $_{S'}T'_R$, the following are equivalent:
\begin{itemize}
\item[(i)] $T{\sim}T'$.
\item[(ii)] there exist two invertible bimodules $_SP_{S'}$ and $_{S'}Q_S$ such that $T_R\cong \Hom_{S'}(Q,T')$ and
$T'_R\cong \Hom_S(P,T)$, as right $R$-modules.
\end{itemize}
\end{lemma}

\begin{proof}
Let $A$ and $B$ be two rings, and let $_B X_A$ be an invertible bimodule. By Lemma \ref{4.6}, there exists an invertible
bimodule $_A Y_B$ such that the functor $X \otimes_A - : A\text{-Mod} \longrightarrow B\text{-Mod}$ is naturally equivalent
to the functor $\Hom_A(Y, -) : A\text{-Mod} \longrightarrow B\text{-Mod}$, and similarly, the functor $\Hom_B(X,-):B\text{-Mod}
\longrightarrow A\text{-Mod}$ is naturally equivalent to the functor $Y \otimes_B - : B\text{-Mod} \longrightarrow A\text{-Mod}$.
Hence, statements (i) and (ii) are equivalent.
\end{proof}

We will use the following three lemmas repeatedly in Sections 5 and 6.

\begin{lemma}\label{4.14}
Let $_ST_R$ and $_{S'}T'_R$ be two bimodules. If $_{S}T_R ~{\sim} ~_{S'}{T'}_R$, then
\begin{itemize}
\item[(i)] $\Prod(T_R)=\Prod(T'_R)$.
\item[(ii)] $\Add(T_R)=\Add(T'_R)$.
\item[(iii)] $\add(T_R)=\add(T'_R)$.
\item[(iv)] $~^\perp T=~^\perp T'$ and $T^\perp=T'^\perp$.
\item[(v)] $\Gen^*(T)=\Gen^*(T')$ and $\gen^*(T)=\gen^*(T')$.
\item[(vi)] $\Cogen^*(T)=\Cogen^*(T')$ and $\cogen^*(T)=\cogen^*(T')$.
\end{itemize}
\end{lemma}

\begin{proof}
Assume that $_{S}T_R ~{\sim} ~_{S'}{T'}_R$. Then, by definition, there are invertible bimodules $_SP_{S'}$ and
$_{S'}Q_S$, such that $T\cong P\otimes_{S'}T'$ and $T'\cong Q\otimes_ST $, as right $R$-modules.

Since the proofs of (i), (ii) and (iii) are similar, we only prove (i). By the symmetry, it suffices to show that
$\Prod(T_R)\subseteq \Prod(T'_R)$. Let $X\in \Prod(T_R)$. Then there exist a module $Y_R$ and a set $I$ such that
$X\oplus Y\cong T^I$. So, we have
$$X\oplus Y\cong T^I\cong (P\otimes_{S'}T')^I.~~~~\ (\dagger)$$ On the other hand, since $P$ is a finitely generated
projective right $S'$-module, there exist a module $Z_{S'}$ and a positive integer $\ell$ such that $P\oplus
Z\cong S'^{\ell}$. By $(\dagger)$, we conclude the following isomorphism of right $R$-modules:
\[\begin{array}{lllll}
X\oplus Y \oplus (Z\otimes_{S'}T')^I & \cong  (P\otimes_{S'}T')^I\oplus (Z\otimes_{S'}T')^I\\
& \cong  ((P\otimes_{S'}T')\oplus (Z\otimes_{S'}T'))^I\\
& \cong  ((P\oplus Z)\otimes_{S'}T')^I\\
& \cong  (S'^{\ell}\otimes_{S'}T')^I\\
& \cong  (T'^l)^I. \\
\end{array}\]
Hence, $X\in \Prod(T'_R)$.

(iv) There exist a right $S'$-module $P'$, a right $S$-module $Q'$, and positive integers $\ell$ and $k$ such that
${S'}^{\ell}\cong P \oplus P'$ and $S^{k}\cong Q \oplus Q'$. From these, we deduce the following isomorphism of right
$R$-modules: $$T'^{\ell} \cong T \oplus (P' \otimes_{S'} T')$$ and $$T^{k}\cong T' \oplus (Q' \otimes_S T).$$ Thus,
for any module $X_R$ and every integer $i \geq 0$, $\Ext_R^i(X, T)$ is a direct summand of $\Ext_R^i(X, T')^{\ell}$,
and $\Ext_R^i(X, T')$ is a direct summand of $\Ext_R^i(X, T)^{k}$. Thus, $~^\perp T_R= ~^\perp T'_R$. A similar
argument shows that ${T_R}^\perp={T'_R}^\perp$.

The argument given in the proof of (iv) also shows that $~^{{\perp}_1} T = ~^{{\perp}_1} T'$ and $T^{{\perp}_1} =
T'^{{\perp}_1}$. Thus, (v) and (vi) follow directly from (i), (ii), and (iii).
\end{proof}

We say that the Krull-Schmidt theorem holds for finitely generated right $R$-modules if every nonzero finitely generated
module $M_R$ admits a finite decomposition into indecomposable modules, $M\cong M_1\oplus \cdots \oplus M_n$, such that for 
every indecomposable decomposition $M\cong N_1\oplus \cdots \oplus N_k$, we have $k=n$ and there exists a permutation $\sigma$ 
of $\{1,\ldots,n\}$ such that $N_i\cong M_{\sigma(i)}$ for each $i=1,\ldots,n$. If $R$ is a complete algebra, by 
\cite[Theorem 21.35]{L}, or if $R$ is a right Artinian ring, by \cite[Theorem 12.9]{AF}, the Krull-Schmidt theorem holds
for finitely generated right $R$-modules.

Next, we show that if the Krull–Schmidt theorem holds for finitely generated right $R$-modules, then two basic
Wakamatsu tilting modules $_ST_R$ and $_{S'}T'_R$ are equivalent if and only if they are isomorphic. Recall that
a nonzero module $X_R$ is said to be {\it basic} if it admits a decomposition $X \cong X_1 \oplus \cdots \oplus
X_n$ such that each $X_i$ is indecomposable and $X_i \ncong X_j$ whenever $i \neq j$.

\begin{lemma}\label{4.15}
Assume that the Krull-Schmidt theorem holds for finitely generated right $R$-modules (e.g. $R$ is a complete algebra or
a right Artinian ring). Let $_ST_R$ and $_{S'}T'_R$ be two bimodules that are finitely generated and basic as right
$R$-modules. Then $T{\sim} T'$ if and only if $T\cong T'$, as right $R$-modules.
\end{lemma}

\begin{proof}
Clearly if $T\cong T'$, then $T{\sim} T'$.

Conversely, suppose that $T{\sim} T'$. Then, there exist two invertible bimodules $_SP_{S'}$ and $_{S'}Q_S$ such that
$T\cong P\otimes_{S'}T'$ and $T'\cong Q\otimes_ST$, as right $R$-modules. There exist a module $Z_{S'}$ and a positive
integer $\ell$ such that $S'^{\ell}\cong P\oplus Z$. From this, we conclude that $$T'^{\ell}\cong T\oplus
(Z\otimes_{S'}T'),\ \  (\dag)$$ as right $R$-modules. There exist positive integers $m, n$ and indecomposable right
$R$-modules $T_1, T_2, \dots, T_m$ and $T'_1, T'_2, \dots, T'_n$ such that $$T \cong \bigoplus_{i=1}^m T_i \quad \text{and}
\quad T' \cong \bigoplus_{i=1}^n T'_i,$$ where $T_i\ncong T_j$ and $T'_i \ncong T'_j$ for all distinct integers $i$ and $j$.
Since the Krull–Schmidt theorem holds for finitely generated right $R$-modules, by $(\dag)$, we conclude that $m\leq n$ and
that each $T_i$ is isomorphic to some $T'_j$. Since the relation ${\sim}$ is symmetric, we also obtain $n\leq m$. Therefore,
$m=n$, and hence $T\cong T'$.
\end{proof}

We end this section by establishing a variant of Lemma \ref{4.14} for a finitely generated module $T_R$ and a basic direct
summand $T'_R$ of $T$.

\begin{lemma}\label{4.14a}
Assume that the Krull–Schmidt theorem holds for finitely generated right $R$-modules. Let $T_R$ be a finitely generated
module. Then $T$ has a basic direct summand $T'$, unique up to isomorphism, such that:
\begin{itemize}
\item[(i)] $\Prod(T_R)=\Prod(T'_R)$.
\item[(ii)] $\Add(T_R)=\Add(T'_R)$.
\item[(iii)] $\add(T_R)=\add(T'_R)$.
\item[(iv)] $~^\perp T=~^\perp T'$ and $T^\perp=T'^\perp$.
\item[(v)] $\Gen^*(T)=\Gen^*(T')$ and $\gen^*(T)=\gen^*(T')$.
\item[(vi)] $\Cogen^*(T)=\Cogen^*(T')$ and $\cogen^*(T)=\cogen^*(T')$.
\item[(vii)] $T_R$ is Wakamatsu tilting if and only $T'_R$ is Wakamatsu tilting.
\item[(viii)] $T_R$ is Wakamatsu cotilting if and only $T'_R$ is Wakamatsu cotilting.
\end{itemize}
\end{lemma}

\begin{proof}
Since the Krull–Schmidt theorem holds for finitely generated right $R$-modules, there exist indecomposable right $R$-modules
$X_1, \ldots, X_{\ell}$ and positive integers $n_1,\ldots, n_{\ell}$ such that $$T\cong X_1^{n_1}\oplus \cdots \oplus
X_{\ell}^{n_{\ell}}$$ and $X_i\ncong X_j$ whenever $i \neq j$. Set $T'=X_1 \oplus \cdots \oplus X_{\ell}$. Then clearly
$T'$ is basic, $T' \in \add(T_R)$, and $T \in \add(T'_R)$. From this, it is straightforward to see that (i), (ii), (iii), and
(iv) hold. From $T' \in \add(T_R)$ and $T \in \add(T'_R)$, we also readily deduce that $~^{{\perp}_1} T = ~^{{\perp}_1} T'$
and $T^{{\perp}_1} = T'^{{\perp}_1}$. Thus, (v), (vi), (vii) and (viii) follow immediately from (i), (ii), and (iii).
\end{proof}

\section{Proof of Theorems 1.1 and 1.2}

In this section, our objective is to generalize the first Mantese–Reiten theorem \cite[Theorem 2.10]{MR} to associative
rings; see Theorems \ref{5.1} and \ref{5.8}. To prove Theorem \ref{5.1}, we require the following three lemmas. The first
and third lemmas are used, respectively, to verify that the map in Theorem \ref{5.1} is well defined and injective.

\begin{lemma}\label{5.2}
Let $_ST_R$ and $_{S'}T'_R$ be two Wakamatsu tilting bimodules. If $_ST_R~{\sim}~ _{S'}T'_R$, then $\overline{\CX_{T}}=
\overline{\CX_{T'}}$ and $\CX_{T}=\CX_{T'}$.
\end{lemma}

\begin{proof}
The claim follows immediately from Lemma \ref{4.14}.
\end{proof}

\begin{lemma}\label{5.2+}
Let $T_R$ be a module with $S=\End(T_R)$, and let $Y_R$ be a module. Then
\begin{itemize}
\item[(i)] For any $X_R\in \cogen^*(T)$, the natural homomorphism $$\varphi_X:\Hom_R(Y,X)\rt \Hom_S(\Hom_R(X,T),\Hom_R(Y,T)),$$
defined by $(\varphi_X(f))(g)=gf$, is an isomorphism.
\item[(ii)] For any $X_R \in \gen^*(T)$, the natural homomorphism $$\psi_X: \Hom_R(X,Y)\rt \Hom_S(\Hom_R(T,X),\Hom_R(T,Y)),$$
defined by $(\psi_X(f))(g)=fg$, is an isomorphism.
\end{itemize}
\end{lemma}

\begin{proof}
(i) If $X = T^n$, then it is straightforward to see that the natural map $\varphi_X$ is the composition of the
following natural group isomorphisms:
\[\begin{array}{lllll}
\Hom_R(Y,T^n)&\cong \Hom_R(Y,T)^n\\
&\cong \Hom_S(S^n,\Hom_R(Y,T))\\
&\cong \Hom_S(\Hom_R(T^n,T),\Hom_R(Y,T))\\
&\cong \Hom_S(\Hom_R(X,T),\Hom_R(Y,T)).
\end{array}\]
According to the above isomorphisms, we can conclude that $\varphi_X$ is an isomorphism for any $X\in \add(T_R)$.

Let $X \in \cogen^*(T_R)$. Then, there exists an exact sequence $$0 \rt X \overset{f}\rt T_0 \overset{g} \rt T_1$$
with $T_0, T_1 \in \add(T_R)$, $\Ext_R^1(\cok \ f, T) = 0$ and $\Ext_R^1(\cok \ g, T) = 0$. This induces the
exact sequence $$\Hom_R(T_1, T)\rt \Hom_R(T_0, T) \rt \Hom_R(X, T) \rt 0.$$ For any module $Z_R$, set $Z^T=
\Hom_R(Z, T)$. Applying the left exact functor $\Hom_S(-, Y^T)$ to the above sequence yields the bottom row in
the following commutative diagram with exact rows:
\vspace{0.2cm}
$$\begin{CD}
0 @>>> \Hom_R(Y,X) @>>> \Hom_R(Y,T_0)@>>> \Hom_R(Y,T_1) \\
& & @V\varphi_X VV @V\varphi_{T_0} VV @V\varphi_{T_1} VV & & & &\\ 0 @>>> \Hom_S(X^T,Y^T) @>>>
\Hom_S({T_0}^T,Y^T) @>>> \Hom_S({T_1}^T,Y^T)
\end{CD}$$
\vspace{0.5cm}
As $\varphi_{T_0}$ and $\varphi_{T_1}$ are isomorphisms, so  $\varphi_X$ is also an isomorphism.

(ii) The proof is the dual of the proof of (i).
\end{proof}

\begin{lemma}\label{5.3}
Let $_ST_R$ and $_{S'}T'_R$ be two Wakamatsu tilting bimodules. If either $\overline{\CX_{T}}=\overline{\CX_{T'}}$ or
$\CX_{T}=\CX_{T'}$, then $_ST_R~{\sim}~ _{S'}T'_R$.
\end{lemma}

\begin{proof}
Lemma \ref{2.6}(v) yields that $\overline{\CX_{L}}\medcap \gen^*(R) = \CX_{L}$ for any Wakamatsu tilting module $L_R$.
Hence, it suffices to show that the equality $\CX_{T} = \CX_{T'}$ implies $_S T_R \sim {}_{S'}T'_R$.

Assume that $\CX_{T} = \CX_{T'}$. In particular, $T \in \overline{\CX_{T'}}$ and $T' \in \overline{\CX_{T}}$, which, by
Lemma \ref{2.6}(i), yield the isomorphisms $$T\cong \Hom_R(T',T) \otimes_{S'} T'$$ and $$T'\cong \Hom_R(T,T')\otimes_S T,$$
as right $R$-modules. The proof is therefore complete once we verify that $_S P_{S'} = \Hom_R(T', T)$ and $_{S'} Q_S=
\Hom_R(T, T')$ are invertible bimodules. Equivalently, by Theorem \ref{4.9}, it is enough to prove that the bimodules
$_SP_{S'}$ and $_{S'}Q_S$ are projective of rank one. We provide only the argument for ${}_S P_{S'}$, since the other
case is similar.
	
Since $T'\in \CX_{T}$, by Lemma \ref{2.6a}(iii) there exists an exact sequence $$0 \rt Y \rt T_0 \rt T' \rt 0$$ with
$T_0 \in \add(T_R)$ and $Y \in \CX_{T}$. As $\CX_{T}=\CX_{T'}\subseteq T'^{\perp}$, this sequence splits, and hence
$T'\in \add(T_R)$. By Lemma \ref{2.11}(ii), it follows that $\Hom_R(T',T)$ is a finitely generated projective left
$S$-module. Similarly, we have $T \in \add(T'_R)$, and therefore $\Hom_R(T',T)$ is a finitely generated projective
right $S'$-module by Lemma \ref{2.11}(i).

As $T \in \add(T'_R)$, it follows that $T_R \in \gen^*(T')$. Hence, by Lemma \ref{5.2+}(ii), there is a natural isomorphism
$$\Hom_R(T,T)\cong \Hom_{S'}(\Hom_R(T',T),\Hom_R(T',T)).$$ A straightforward verification shows that this coincides with
the natural map $\eta_{P_{S'}}:S\rt \End(P_{S'})$. Similarly, applying Lemma \ref{5.2+}(i), we see that the natural map 
$\eta_{_SP}:S'\rt \End(_SP)$ is an isomorphism as well.
\end{proof}

\begin{theorem}\label{5.1}
The map $[T]\mapsto \overline{\mathcal{X}_{T}}$ establishes a one-to-one correspondence between the equivalence classes
of Wakamatsu tilting right $R$-modules and preenveloping $\overline{coresolving}$ subcategories of $\text{Mod-}R$ with
an Ext-projective $\overline{generator}$ in $\text{gen}^*(R)$, maximal among those with the same Ext-projective
$\overline{generator}$.
\end{theorem}

\begin{proof}
This map is well defined according to Lemma \ref{2.6} and Lemma \ref{5.2}. It is injective by Lemma \ref{5.3}, and it is
surjective by \cite[Proposition 3.6]{S}.
\end{proof}

We record the following corollary, which also extends \cite[Theorem 2.10]{MR}.

\begin{corollary}\label{5.4}
Assume that the Krull-Schmidt theorem holds for finitely generated right $R$-modules. Then the map $\phi:T \mapsto
\overline{\CX_{T}}$ establishes a one-to-one correspondence between the isomorphism classes of basic Wakamatsu tilting
right $R$-modules and the preenveloping $\overline{coresolving}$ subcategories of $\text{Mod-}R$ that possess an
Ext-projective $\overline{generator}$ in $\gen^*(R)$ and are maximal among those having the same Ext-projective
$\overline{generator}$.
\end{corollary}

\begin{proof}
By Theorem \ref{5.1} and Lemma \ref{4.15}, this map is well defined and injective. Let $\mathcal{C}$ be a preenveloping
$\overline{coresolving}$ subcategory of $\text{Mod-}R$ that possesses an Ext-projective $\overline{generator}$ in
$\gen^*(R)$ and is maximal among all subcategories of $\text{Mod-}R$ having the same Ext-projective $\overline{generator}$.
By Theorem \ref{5.1}, there exists a Wakamatsu tilting module $T_R$ such that $\mathcal{C}=\overline{\mathcal{X}_{T}}$.
By Lemma \ref{4.14a}, there exists a basic Wakamatsu tilting module $T'_R$ such that $\overline{\mathcal{X}_{T'}}=
\overline{\mathcal{X}_{T}}$, and thus $\phi$ is surjective.
\end{proof}

We conclude this section by proving another generalization of \cite[Theorem 2.10]{MR}, which serves as an analogue of
Theorem \ref{5.1} for the category $\text{mod-}R$ when $R$ is a right Morita ring. To this end, we begin by recalling
the definition of a right Morita ring, followed by a lemma that will be essential in the proof.

\begin{definition}\label{5.6}
A ring $R$ is called {\it right Morita} if there exist a ring $S$ and a bimodule $_S U_R$ such that the functors
$\Hom_R(-,U)\colon \text{mod-}R \to S\text{-mod}$ and $\Hom_S(-,U)\colon S\text{-mod}\to \text{mod-}R$ induce a
duality of categories.
\end{definition}

It is easy to see that every Artin algebra is a right Morita ring. Conversely, right Morita rings need not be
Artin algebras; see \cite[Example 3.6]{An1}. Moreover, \cite[Corollary 24.9]{AF} shows that any right Morita
ring is right Artinian, and hence also right Noetherian.

Let $R$ be a right Morita ring with a bimodule $_S U_R$ as in the above definition. By \cite[Corollary 24.9]{AF},
$U_R$ is a finitely generated injective cogenerator. Let $M_R$ be a finitely generated module. Then there exist
a positive integer $n$ and a surjective homomorphism $\phi: S^n\to \Hom_R(M,U)$ of finitely generated left
$S$-modules. Applying the functor $\Hom_S(-,U)$ to $\phi$ induces an injective homomorphism $\psi: M \to U^n$ of
finitely generated right $R$-modules. Hence, every finitely generated right $R$-module embeds into a finitely
generated injective right $R$-module.

\begin{lemma}\label{5.7}
Let $R$ be a right Morita ring and $\CX$ a coresolving subcategory of $\text{mod-}R$. If $\CX$ admits an
Ext-projective generator $T_R$, then $T_R$ is a Wakamatsu tilting module.
\end{lemma}

\begin{proof}
Since $T_R$ is an Ext-projective generator for $\CX$, it is self-orthogonal. On the other hand, because $R$ is
right Noetherian and $T_R$ is finitely generated, we have $T_R \in \gen^*(R)$. Hence, it only remains to show
that $R_R\in \cogen^*(T)$.

We claim that for any $X \in {}^{\perp}\CX \medcap \text{mod-}R$, there exists an exact sequence $$0 \to X \to Y
\to Z \to 0,$$ with $Y \in \add(T_R)$ and $Z \in {}^{\perp}\CX \medcap \text{mod-}R$. Applying this claim to $X=R$
shows that $R_R \in \cogen^*(T)$.

Because $R$ is right Artinian, by \cite[Theorem 12.9]{AF} the Krull-Schmidt theorem holds for finitely generated
right $R$-modules. Using this, we can conclude that $\add(T_R)$ contains only finitely many isomorphism classes
of indecomposable modules. Thus, by \cite[Proposition 4.2]{AS2}, every finitely generated right $R$-module admits
an $\add(T_R)$-preenvelope.

Let $X \in {}^{\perp}\CX \medcap \text{mod-}R$ and let $f \colon X \to Y$ be an $\add(T_R)$-preenvelope. Since $R$
is a right Morita ring and $X$ is finitely generated, there exists a finitely generated injective module $I_R$ and a
monomorphism $g \colon X \to I$. As $T_R$ is an Ext-projective generator for $\CX$ and $I \in \CX$, there is an
exact sequence $$0 \to K \to T_0 \xrightarrow{\varphi} I \to 0,$$ where $T_0 \in \add(T_R)$ and $K \in \CX$. Because
$X \in {}^{\perp}\CX$, we have $\Ext_R^1(X, K) = 0$, so the map $$\Hom_R(X,\varphi)\colon \Hom_R(X, T_0) \to
\Hom_R(X, I)$$ is surjective.
Thus, there exists $h \colon X \to T_0$ such that $\varphi h = g$, and consequently $h$ is monic. Hence $f$ is also
monic, giving the exact sequence $$0 \to X \xrightarrow{f} Y \to Z \to 0. \quad (\dag)$$ Clearly, $Z$ is finitely
generated. It remains to show that $Z \in {}^{\perp}\CX$. Thus, we must prove that $Z \in {}^{\perp}N$ for every
$N \in \CX$.

Let $N \in \CX$. Since $X\in {}^{\perp}\CX$, we have $\Ext_R^{i\ge 1}(X,N)=0$ for all $i \ge 1$. Moreover, because
$T_R$ is Ext-projective in $\CX$ and $Y \in \add(T_R)$, it follows that $\Ext_R^{i\ge 1}(Y,N)=0$ for all $i \ge 1$.
Consequently, by $(\dag)$, we conclude that $\Ext_R^i(Z,N)=0$ for all $i \ge 2$. Therefore, to prove that $Z\in
{}^{\perp}N$, it suffices to show that $$\Hom_R(f,N):\Hom_R(Y, N) \to \Hom_R(X, N)$$is surjective.

Consider an exact sequence  $$0 \to L \to T_1 \xrightarrow{\psi} N \to 0,$$ with $T_1 \in \add(T_R)$ and $L \in \CX$.
Since $X \in {}^{\perp}\CX$, we have $\Ext_R^1(X,L)=0$, so the map  $$\Hom_R(X,\psi):\Hom_R(X,T_1)\to \Hom_R(X,N)$$
is surjective. For any $\theta \in \Hom_R(X,N)$, choose a $\rho \in \Hom_R(X,T_1)$ with $\psi \rho=\theta$. As $f$ is
an $\add(T_R)$-preenvelope and $T_1 \in \add(T_R)$, the map  $$\Hom_R(f,T_1): \Hom_R(Y, T_1)\to \Hom_R(X,T_1)$$
is surjective. Hence, there exists $\gamma \in \Hom_R(Y,T_1)$ such that $\gamma f=\rho$, and then $(\psi \gamma) f=\theta$
showing that $\Hom_R(f,N)$ is surjective.
\end{proof}

\begin{theorem}\label{5.8}
Let $R$ be a right Morita ring. Then the map $\phi : T \to \CX_{T}$ establishes a bijection between the isomorphism
classes of basic Wakamatsu tilting right $R$-modules and the coresolving subcategories of $\text{mod-}R$ that possess
an Ext-projective generator and are maximal among those having the same Ext-projective generator.
\end{theorem}

\begin{proof}
This map is well defined by Lemma \ref{2.6a} and injective by Lemmas \ref{4.15} and \ref{5.3}. It thus remains to
show that $\phi$ is surjective.

Let $\mathcal{X}$ be a coresolving subcategory of $\text{mod-}R$ that has an Ext-projective generator $T_R$ and is
maximal among those subcategories of $\text{mod-}R$ having $T$ as their Ext-projective generator. As $R$ is right
Artinian, by \cite[Theorem 12.9]{AF}, the Krull–Schmidt theorem holds for finitely generated right $R$-modules.
Hence, by Lemma \ref{4.14a}, $T$ has a unique basic direct summand $T'$ such that $\add(T_R) = \add(T'_R)$ and
$T^\perp=T'^\perp$. It follows that, for any subcategory $\mathcal{Z}$ of $\text{mod-}R$, $T$ is an Ext-projective
generator for $\mathcal{Z}$ if and only if $T'$ is. Therefore, without loss of generality, we may and do assume that
$T$ is basic.

Lemma \ref{5.7} implies that $T_R$ is a basic Wakamatsu tilting module. Since $T_R$ is an Ext-projective generator
for $\mathcal{X}$, it follows that $\mathcal{X} \subseteq T^\perp \medcap \gen^*(T) = \CX_{T}$. By the maximality
assumption on $\mathcal{X}$, we have $\mathcal{X} = \CX_{T}$, and hence $\phi$ is surjective.
\end{proof}

\section{Proof of Theorems 1.3, 1.4 and 1.5}

In this section, we extend the second Mantese–Reiten theorem \cite[Theorem 2.12]{MR} to associative rings; see Theorems
\ref{6.9}, \ref{6.1} and \ref{6.1a}. To establish that the map in Theorem \ref{6.9} is well defined, injective, and
surjective, we rely on Lemmas \ref{6.2}, \ref{6.4}, and \ref{6.7}, respectively.

\begin{lemma}\label{6.2}
Let $_ST_R$ and $_{S'}T'_R$ be two bimodules. If $_{S}T_R ~{\sim} ~_{S'}{T'}_R$, then ${\overline{\CY_{T}}}=
{\overline{\CY_{T'}}}$ and ${\mathcal{Y}_{T}}={\mathcal{Y}_{T'}}$.
\end{lemma}

\begin{proof}
It follows immediately by Lemma \ref{4.14}.
\end{proof}

\begin{lemma}\label{6.3}
Let $T_R$ be a module with $S=\End(T_R)$. For any $X_R\in \cogen^*(T)$, the natural homomorphism $\pi_X: X\rt
\Hom_S(\Hom_R(X,T),T)$ is an isomorphism.
\end{lemma}

\begin{proof}
It holds by \cite[Proposition 1]{W2}.
\end{proof}

\begin{lemma}\label{6.4}
Let $T_R$ and $T'_R$ be two finitely generated self-orthogonal modules. Assume that $S=\End(T_R)$ and $S'=\End(T'_R)$.
If $\CY_{T}\medcap \text{mod-}R=\CY_{T'}\medcap \text{mod-}R$, then ${}_S T_R \sim {}_{S'} T'_R$.
\end{lemma}

\begin{proof}
Assume that ${\mathcal{Y}_{T}}\medcap \text{mod-}R={\mathcal{Y}_{T'}}\medcap \text{mod-}R$. Then $T\in {\mathcal{Y}_{T'}}$
and $T'\in {\mathcal{Y}_{T}}$, and so by Lemma \ref{6.3}, we have the following isomorphisms of right $R$-modules: $$T\cong \Hom_{S'}(\Hom_R(T,T'),T')$$ and $$T'\cong \Hom_S(\Hom_R(T',T),T).$$ In view of Lemma \ref{4.13}, the argument is complete
if we can show that $_SP_{S'}=\Hom_R(T',T)$ and $_{S'}Q_S=\Hom_R(T,T')$ are invertible bimodules. Equivalently, by Theorem
\ref{4.9}, it suffices to show that the bimodules $_SP_{S'}$ and $_{S'}Q_S$ are projective of rank one. We will prove this
only for $_{S'}Q_S$, as the proof of the other case is similar.

Since $T\in \CY_{T'}$, by Lemma \ref{2.10}(iii), there exists an exact sequence $$0\rt T\rt T_0\rt Y \rt 0$$ in
which $T_0\in \add(T'_R)$ and $Y\in {\mathcal{Y}_{T'}}$. As $\CY_{T'}\medcap \text{mod-}R=\CY_{T}\medcap
\text{mod-}R\subseteq ^{\perp}T$, this sequence splits, and so $T\in \add(T'_R)$. Hence, by Lemma \ref{2.11}(ii),
$\Hom_R(T,T')$ is a finitely generated projective left $S'$-module. Similarly, we get $T'\in \add(T_R)$, and so
$\Hom_R(T,T')$ is a finitely generated projective right $S$-module by Lemma \ref{2.11}(i).

As $T\in \add(T'_R)$, it follows that $T_R\in \cogen^*(T')$. Therefore, by Lemma \ref{5.2+}(i), we obtain a natural
isomorphism $$\Hom_R(T,T)\cong \Hom_{S'}(\Hom_R(T,T'),\Hom_R(T,T')).$$ A direct verification shows that this is
precisely the natural map $\eta_{_{S'}Q}: S\rt \End(_{S'}Q)$. Likewise, applying Lemma \ref{5.2+}(ii), we conclude
that the natural map $\eta_{Q_S}: S'\rt \End(Q_S)$ is an isomorphism as well. Thus, the bimodule $_{S'}Q_S$ is
projective of rank one.
\end{proof}

\begin{lemma}\label{6.7}
Let $\CC$ be a precovering $\overline{resolving}$ subcategory of $\text{Mod-}R$ with an Ext-injective
$\overline{cogenerator}$ $C_R$. Then $\CC^{\perp}\subseteq \gen^*(\CC)$. Moreover, if either
$\Prod(C_R) \subseteq \CC$ or $C_R$ is product-complete, then $C_R$ is a Wakamatsu cotilting module.
\end{lemma}

\begin{proof}
Since $C \in \CC^{\perp} \medcap \CC$, it follows that $C_R$ is self-orthogonal. As $\CC$ contains all projective
right $R$-modules, any $\CC$-precover map is surjective.

Let $N_R$ be a module in $\CC^{\perp}$, and let $g: L \rt N$ be a $\CC$-precover of $N$. Since $C$ is
an Ext-injective $\overline{cogenerator}$ for $\CC$, there exists an exact sequence $$0\longrightarrow
L\st{\alpha}\longrightarrow C_0\longrightarrow Y\longrightarrow 0 \  \ (\dag)$$ with $C_0 \in \Prod(C_R)$
and $Y \in \CC$. As $\CC$ is closed under extensions, $(\dag)$ implies that $C_0 \in \CC$. Moreover, since
$C \in \CC^{\perp}$ and $C_0 \in \Prod(C_R)$, we also have $C_0 \in \CC^{\perp}$.

Because $N \in \CC^{\perp}$, applying $\Hom_R(-,N)$ to $(\dag)$ yields a map $f_0 : C_0 \rt N$ such that
$g=f_0 \alpha$. It is straightforward to verify that $f_0$ is also a $\CC$-precover of $N$. Since
$N \in \CC^{\perp}$, $C_0 \in \CC^{\perp}$, and $f_0$ is a $\CC$-precover, the exact sequence
 $$0\longrightarrow \Ker \ f_0\longrightarrow C_0 \overset{f_0}\longrightarrow N
\longrightarrow 0,$$ shows that $\Ker \ f_0 \in \CC^{\perp}$. Repeating this process yields an exact sequence
$$\cdots \longrightarrow C_i\st{f_i}\longrightarrow C_{i-1}\st{f_{i-1}}\longrightarrow
\cdots\rt C_1\st{f_1}\longrightarrow C_0 \st{f_0}\longrightarrow N\longrightarrow 0$$such that $C_i \in
\Prod(C_R) \medcap \CC$ and $\Ker \ f_i \in \CC^{\perp}$ for all $i \ge 0$. Thus $N \in \gen^*(\CC)$, and so
$\CC^{\perp}\subseteq \gen^*(\CC)$.

Let $Q_R$ be an injective cogenerator for $\text{Mod-}R$. Clearly, $Q \in \CC^{\perp}$. By what was shown
above, we obtain an exact sequence  $$\cdots \longrightarrow C_i\st{f_i}\longrightarrow C_{i-1}\st{f_{i-1}}
\longrightarrow \cdots\rt C_1\st{f_1}\longrightarrow C_0 \st{f_0}\longrightarrow Q\longrightarrow 0 \  \
(\ddag)$$ with $C_i\in \Prod(C_R)\medcap \CC$ and $\Ker \ f_i \in \CC^{\perp}$ for all $i \ge 0$.

Suppose first that $\Prod(C_R) \subseteq \CC$. Then $\CC^{\perp} \subseteq \Prod(C_R)^{\perp}$, and
from $(\ddag)$ we conclude that $Q \in \gen^*(\Prod(C_R))$ and $\Ext^{i\ge 1}_R(C^I,C)=0$ for every set $I$.

Next, assume that $C_R$ is product-complete. Then $\Prod(C_R) = \Add(C_R)$. For any $i \ge 0$, since
$C \in \CC$, we have $\Ker \ f_i \in C^{\perp}$, and hence $$\Ker \ f_i \in \Add(C_R)^{\perp} = \Prod(C_R)^{\perp}.
$$ Thus again $Q \in \gen^*(\Prod(C_R))$ and $\Ext^{i\ge 1}_R(C^I,C)=0$ for every set $I$. Therefore, if either
$\Prod(C_R) \subseteq \CC$ or $C_R$ is product-complete, then $C_R$ is a Wakamatsu cotilting module.
\end{proof}

\begin{theorem}\label{6.9}
The map $[C] \mapsto {\overline{\CY_{C}}}$ establishes a one-to-one correspondence between the equivalence classes
of product-complete Wakamatsu cotilting modules in $\text{mod-}R$ and the precovering $\overline{resolving}$
subcategories of $\text{Mod-}R$ that admit a product-complete Ext-injective $\overline{\text{cogenerator}}$ in
$\text{mod-}R$, maximal among those sharing the same Ext-injective $\overline{\text{cogenerator}}$.
\end{theorem}

\begin{proof}
This map is well defined by Lemmas \ref{2.9} and \ref{6.2} and surjective by Lemma \ref{6.7}. It therefore remains
to prove that it is injective.

For any product-complete Wakamatsu cotilting module $C_R\in \text{mod-}R$, it follows from \cite[Lemma 3.13]{DMT} that $${\overline{\CY_{C}}}\medcap \text{mod-}R={{\CY_{C}}}\medcap \text{mod-}R.$$

Let $C_1$ and $C_2$ be two product-complete Wakamatsu cotilting modules in $\text{mod-}R$ such that ${\overline{\CY_{C_1}}}
= {\overline{\CY_{C_2}}}$. Intersecting both sides with $\text{mod-}R$ yields  $$\CY_{C_1}\medcap \text{mod-}R=\CY_{C_2}
\medcap \text{mod-}R.$$ By Lemma \ref{6.4}, we therefore conclude that $C_1 \sim C_2$.
\end{proof}

In view of Lemmas \ref{4.15} and \ref{4.14a} and Theorems \ref{3.6} and \ref{6.9}, the following is immediate.

\begin{corollary}\label{6.10}
Let $R$ be a complete algebra. The map $T\mapsto {\overline{\CY_{T}}}$ establishes a one-to-one correspondence
between the isomorphism classes of basic Wakamatsu tilting right $R$-modules of finite length and the precovering
$\overline{resolving}$ subcategories of $\text{Mod-}R$ that admit a finite length Ext-injective $\overline{cogenerator}$,
and that are maximal among those with the same Ext-injective $\overline{\text{cogenerator}}$.
\end{corollary}

\begin{theorem}\label{6.1}
Let $R$ be a right Noetherina ring. The map $[T]\mapsto {\mathcal{Y}_{T}}$ establishes a one-to-one correspondence
between the equivalence classes of Wakamatsu tilting right $R$-modules and the resolving subcategories of
$\text{mod-}R$ with an Ext-injective cogenerator, maximal among those with the same Ext-injective cogenerator.
\end{theorem}

\begin{proof}
This map is well defined by Lemmas \ref{2.10} and \ref{6.2}. It is injective by Lemma \ref{6.4}, and so it
remains to show that it is surjective.

Let $\CC$ be a resolving subcategory of $\text{mod-}R$, and let $T_R$ be an Ext-injective cogenerator for
$\CC$. Assume that $\CC$ is maximal among all subcategories of $\text{mod-}R$ having $T_R$ as their
Ext-injective cogenerator. Since $T_R$ is finitely generated and $R$ is right Noetherian, it follows that
$T_R\in \gen^*(R)$. As $T_R$ is Ext-injective in $\CC$, it follows that $T_R$ is self-orthogonal. Since
$T_R$ also serves as an Ext-injective cogenerator for $\CC$ and $R\in \CC$, we can construct an exact sequence
$$0\rt R\rt T_0\rt L\rt 0,$$ where $T_0\in \add(T_R)$ and $L\in \CC$. The Ext-injectivity of $T_R$ in $\CC$,
combined with $L \in \CC$, implies that $\Ext^1_R(L,T_R)=0$. By repeating this procedure, we establish that
$R_R\in \cogen^*(T_R)$. Thus, $T_R$ is a Wakamatsu tilting module, and so by Lemma \ref{2.10}(iii), we conclude
that ${\mathcal{Y}_{T}}=\CC$.
\end{proof}

By Theorem \ref{6.1} and Lemmas \ref{4.15} and \ref{4.14a}, we obtain the following consequence, which extends
\cite[Theorem 2.12]{MR}.

\begin{corollary}\label{6.6}
Assume that the Krull-Schmidt theorem holds for finitely generated right $R$-modules. Then $\phi:T\mapsto \CY_{T}$
is a bijections between the isomorphism classes of basic Wakamatsu tilting right $R$-modules and the resolving
subcategories of $\text{mod-}R$ with an Ext-injective cogenerator, maximal among those with the same Ext-injective
cogenerator.
\end{corollary}

The dual of the proof of Lemma \ref{5.7} yields the following:

\begin{lemma}\label{6.5}
Let $R$ be a right Morita ring and $\CC$ be a resolving subcategory of $\text{mod-}R$. If $\CC$ admits a product-complete
Ext-injective cogenerator $C_R$, then $C_R$ is a Wakamatsu cotilting module.
\end{lemma}

The above result helps ensure that the dual proof of Theorem \ref{5.8} works in establishing our final result.

\begin{theorem}\label{6.1a}
Let $R$ be a right Morita ring. The map $C\mapsto {\mathcal{Y}_{C}}$ establishes a one-to-one correspondence between the
isomorphism classes of basic product-complete Wakamatsu cotilting modules in $\text{mod-}R$ and the resolving subcategories
of $\text{mod-}R$ with a product-complete Ext-injective cogenerator, maximal among those with the same Ext-injective cogenerator.
\end{theorem}



\begin{thebibliography}{9999}

\bibitem{AF}{F. W. Anderson and  K. R. Fuller}, {\sl Rings and categories of modules}, (2nd ed.), Graduate Texts in Mathematics,
{\bf 13}, New York, Springer-Verlag, (1992).

\bibitem{An}{L. Angeleri H\"{u}gel}, {\sl Infinite dimensional tilting theory}, Advances in representation theory of algebras,
EMS Series of Congress Reports, Zürich, (2013), 1–37.

\bibitem{An1}{L. Angeleri H\"{u}gel}, {\sl Finitely cotilting modules}, Commun. Algebra, {\bf 28}(4), (2000), 2147-2172.

\bibitem{ASS}{I. Assem, D. Simson and A. Skowro\'{n}ski}, {\sl  Elements of the representation theory of associative algebras},
vol. 1: Techniques of representation theory, London Mathematical Society Student Texts, {\bf 65}, Cambridge University Press,
(2006).

\bibitem{A}{M. Auslander}, {\sl Functors and morphisms determined by objects}, Represent. Theory of Algebras, Proc.
Phila. Conf., Lect. Notes pure appl. Math., {\bf 37}, (1978), 1-244.

\bibitem{AR1}{M. Auslander and I. Reiten}, {\sl Applications of contravariantly finite subcategories}, Adv. Math.,
{\bf 86}(1), (1991), 111-152.

\bibitem{AR2}{M. Auslander and I. Reiten}, {\sl Cohen-Macaulay and Gorenstein Artin algebras}, Representation theory
of finite groups and finite-dimensional algebras, Proc. Conf., Bielefeld/Ger. 1991, Prog. Math., {\bf 95}, (1991),
221-245.

\bibitem{ARS}{M. Auslander, I. Reiten and S. O. Smal{\O}}, {\sl Representation theory of Artin algebras}, {\bf 36},
Cambridge Studies in Advanced Mathematics, Cambridge University Press, (1995).

\bibitem{AS1}{M. Auslander and S. O. Smal{\O}}, {\sl Almost split sequences in subcategories}, J. Algebra, {\bf 69}(2),
(1981), 426-454.

\bibitem{AS2}{M. Auslander and S. O. Smal{\O}}, {\sl Preprojective modules over Artin algebras}, J. Algebra, {\bf 66}(1),
(1980), 61-122.

\bibitem{B}{H. Bass}, {\sl Lectures on topics in algebraic K-theory}, Tata Institute of Fundamental Research Lectures
on Mathematics and Physics, Mathematics,  no. {\bf  41}, Bombay: Tata Institute of Fundamental Research, (1967).

\bibitem{BB}{S. Brenner and M. C. R. Butler}, {\sl Generalizations of the Bernstein-Gelfand-Ponomarev reflection functors},
Representation theory, II (Proc. Second Internat. Conf., Carleton Univ., Ottawa, Ont., 1979), Lect. Notes Math.,
{\bf 832}, Springer, Berlin, (1980), 103-169.

\bibitem{BS}{A. B. Buan, {\O}. Solberg}, {\sl  Relative cotilting theory and almost complete cotilting modules}, in: Algebras
and Modules II, in: CMS Conf. Proc., vol. {\bf 24}, Amer. Math. Soc., (1998), 77–92.

\bibitem{DMT}{K. Divaani-Aazar, A. Mahin Fallah and M. Tousi}, {\sl Comparing four definitions of cotilting modules},
preprint; arXiv:2507.18860 [math.RT].

\bibitem{GRS}{E. L. Green, I. Reiten and {\O}. Solberg}, {\sl Dualities on generalized Koszul algebras}, Mem. Am. Math. Soc.,
{\bf 159}(754), (2002), xvi+67 pp.

\bibitem{HR}{D. Happel and C. M. Ringel}, {\sl Tilted algebras}, Trans. Am. Math. Soc., {\bf 274}(2), (1982), 399-443.

\bibitem{H}{Z. Y.  Huang}, {\sl Selforthogonal modules with finite injective dimension III}, Algebr. Represent. Theory,
{\bf 12}(2-5), (2009), 371-384.

\bibitem{KS}{H. Krause and M. Saor\'{\i}n}, {\sl On minimal approximations of modules}, in: E. L. Green, B. Huisgen-Zimmermann
(Eds.), Trends in the Representation Theory of Finite Dimensional Algebras, in: Contemp. Math., {\bf 229}, (1998), 227–236.

\bibitem{L}{T. Y. Lam}, {\sl A first course in noncommutative rings}, Graduate Texts in Mathematics, {\bf 131}, Springer-Verlag,
New York/Berlin, (1991).

\bibitem {MR}{F. Mantese and I. Reiten}, {\sl Wakamatsu tilting modules}, J. Algebra, {\bf 278}(2), (2004), 532-552.

\bibitem{M}{Y. Miyashita}, {\sl Tilting modules of finite projective dimension}, Math. Z., {\bf 193}(1), (1986), 113-146.

\bibitem{Sa}{L. Salce}, {\sl Cotorsion theories for abelian groups}, Symposia Math., {\bf XXIII}, (1979), 11-32.

\bibitem{S}{J. Sun}, {\sl A characterization of Auslander category}, Turk. J. Math., {\bf 37}(5), (2013), 793-805.

\bibitem{W1}{T. Wakamatsu}, {\sl Tilting modules and Auslander’s Gorenstein property}, J. Algebra, {\bf 275}(1), (2004),
3-39.

\bibitem{W2}{T. Wakamatsu}, {\sl On modules with trivial self-extensions}, J. Algebra, {\bf 114}(1), (1988), 106-114.

\end{thebibliography}
\end{document}